\documentclass[12pt]{amsart}
\usepackage{amssymb,epic,eepicemu,xypic,graphicx}
\xyoption{curve}

\newcommand{\w}{\omega}
\renewcommand{\H}{\mathcal{H}}
\newcommand{\E}{\mathcal{E}}
\renewcommand{\P}{\mathcal{P}}
\newcommand{\G}{\Gamma}
\newcommand{\Gbar}{\bar{G}}
\newcommand{\Hbar}{\bar{H}}
\newcommand{\C}{\mathcal{C}}
\newcommand{\V}{\mathcal{V}}

\newcommand{\Ltilde}{\tilde{L}}
\newcommand{\Z}{\mathbb{Z}}
\newcommand{\Sp}{{\rm Sp}}
\newcommand{\PSp}{{\rm PSp}}
\newcommand{\Gal}{\mathop{\rm Gal}}
\newcommand{\Sym}{\mathop{\rm Sym}}
\newcommand{\Out}{\mathop{\rm Out}}
\renewcommand{\a}{\alpha}
\renewcommand{\b}{\beta}
\newcommand{\iso}{\cong}
\newcommand{\sset}{\subseteq}
\newcommand{\PP}{\mathbb{P}}

\newtheorem{theorem}{Theorem}
\newtheorem{lemma}[theorem]{Lemma}

\theoremstyle{remark}
\newtheorem*{remark}{Remark}
\newtheorem*{remarks}{Remarks}

\begin{document}
\title{Monodromy groups of Hurwitz-type problems}
\author{Daniel Allcock}
\address{Department of Mathematics\\University of Texas, Austin}
\email{allcock@math.utexas.edu}
\urladdr{http://www.math.utexas.edu/\textasciitilde allcock}
\author{Chris Hall}
\address{Department of Mathematics\\University of Michigan, Ann Arbor}
\email{hallcj@umich.edu}
%\urladdr{url}
%
\thanks{First author partly supported by NSF grant DMS-0600112.}
\date{March 2, 2008}
%\date{February 22, 2008}
%\date{January 17, 2008}
%\date{January 16, 2008}
%\date{December 11, 2007}
%\date{December 6, 2007}
%\date{October 30, 2007}

\begin{abstract}
We solve the Hurwitz monodromy problem for degree-4 covers.  That is, the
Hurwitz space $\H_{4,g}$ of all simply branched covers of $\PP^1$ of degree
$4$ and genus $g$ is an unramified cover of the space $\P_{2g+6}$ of
$(2g+6)$-tuples of distinct points in $\PP^1$.  We determine the monodromy
of $\pi_1(\P_{2g+6})$ on the points of the fiber.  This turns out to be the
same problem as the action of $\pi_1(\P_{2g+6})$ on a certain local system
of $\Z/2$-vector spaces.  We generalize our result by treating the
analogous local system with $\Z/N$ coefficients, $3\nmid N$, in place of
$\Z/2$.  This in turn allows us to answer a question of Ellenberg
concerning families of Galois covers of $\PP^1$ with deck group
$(\Z/N)^2{:}S_3$.
\end{abstract}

\maketitle

A ramified cover $C$ of $\PP^1$ of degree $d$ is said to have simple
branching if the fiber over every branch point has $d-1$ distinct points. 
Another way to say this is that for each branch point $p$, the permutation
of the sheets of the cover induced by a small loop around $p$ is a
transposition, i.e., a permutation of cycle-shape $21\dots1$.  An Euler
characteristic argument (or the Hurwitz formula) shows that the number of
branch points is $b:=2g+2d-2$, where $g$ is the genus of $C$.  Let
$\H_{d,g}$ be the Hurwitz space, consisting of all such covers, up to
isomorphism as covers.  This is an irreducible smooth algebraic variety. 
There is an obvious map from $\H_{d,g}$ to the space $\P_b$ of unordered
$b$-tuples of distinct points in $\PP^1$.  This is an unramified cover, so
it induces a homomorphism from $G:=\pi_1(\P_b)$ to the symmetric group on
the points of a fiber.  We determine the image in the case $d=4$;    this
answers this case of a question posed explicitly in \cite{EEHS} and
implicit in earlier work.  We call this image $G_2$; the subscript reflects
that this turns out to be the case $N=2$ of a more general construction
considered below.

Our formulation of the problem reflects its topological nature, but
usually one thinks of $\H_{d,g}$ and $\P_b$ as irreducible algebraic
varieties, so that the function field of $\H_{d,g}$ is a finite
extension of that of $\P_b$.  Then $G_2$ is the Galois group of the
associated Galois extension. Even the degree of this extension was unknown.

\begin{theorem}
\label{thm-hurwitz-monodromy}
Let $g>1$.  Then  monodromy group $G_2$ of $\H_{4,g}\to\P_{b=2g+6}$ fits
into the split exact sequence
\begin{equation}
\label{eq-wreath-product}
1
\to
\prod_\Omega\Sp(2g+2,\Z/2)
\to
G_2
\to
\PSp(2g+4,\Z/3)
\to
1,
\end{equation}
where $\Omega=\PP^{2g+3}(\Z/3)$ and $\PSp(2g+4,\Z/3)$ permutes the factors of
the product in the obvious way.
\end{theorem}

\begin{remark}
The $g=0,1$ cases are exceptional.  If $g=0$ then the left term of
\eqref{eq-wreath-product} should be $3^{40}{:}2^{16}$ instead of
$S_3^{40}$, and the sequence is nonsplit.  If $g=1$ then the left term
should be $A_6^{364}{:}2^{168}$ rather than $S_6^{364}$, and we did not
determine whether the sequence splits.  (We use ATLAS notation for group
structures \cite{ATLAS}.)
\end{remark}

The fact that $G_2$ lies in a group fitting into an exact sequence
like \eqref{eq-wreath-product} is due to Eisenbud, Elkies, Harris and Speiser
\cite{EEHS}; see also \cite{Cohen} and \cite{M}.  So our result says
that $G_2$ is as large as possible.  In section~\ref{sec-thm-1} we
will review what we need from \cite{EEHS} and then prove the theorem.

\medskip
In section~\ref{sec-thm-2} we treat two generalizations of this that are
similar to each other.  The degree-4 Hurwitz monodromy problem is very
closely related to a certain local system of $\Z/2$-vector spaces over
$\P_b$.  Namely, $\H_{3,g+1}$ is also an unramified cover of $\P_b$, and
over $\H_{3,g+1}$ there is a universal family $\C_{3,g+1}$ of simply
branched 3-fold covers of $\PP^1$.  (Existence of this family is not hard
to see, and is proven in great generality in \cite{fulton}.)  We write
$\pi$ for the composition $\C_{3,g+1}\to\H_{3,g+1}\to\P_b$.  If $N\geq 0$,
then we consider the sheaf $\V_N:=R^1\pi_*(\Z/N)$ on $\P_b$, which we
recall is the sheaf associated to the presheaf
$U\mapsto H^1(\pi^{-1}(U);\Z/N)$; the case $N=0$ corresponds to $\Z$
coefficients.  $\V_N$ is a local system of $\Z/N$-modules equipped with
symplectic forms; the fiber over a point $p=(p_1,\dots,p_b)\in\P_b$ is
$H^1(\pi^{-1}(p),\Z/N)$, which is the direct sum of the $H^1(C;\Z/N)$ as
$C$ varies over the points of $\H_{3,g+1}$ lying above $p$.   As we explain
in section~\ref{sec-thm-1}, the monodromy of $\pi_1(\P_b)$ on $\V_2$ is
exactly the Hurwitz monodromy group in degree~4, which we called $G_2$.  So
we define $G_N$ as  the monodromy group on $\V_N$.  We have completely
determined $G_N$ when $3\nmid N$, except for the cases $g=0$ or $1$ and the
question of whether the exact sequence \eqref{eq-foo} below splits.

\begin{theorem}
\label{thm-2}
Suppose $3\nmid N$ and  $g\geq0$ ($\,g>1$ if $N$ is even).  Then the
monodromy group $G_N$ of $\V_N$  fits into an exact sequence
\begin{equation}
\label{eq-foo}
1
\to
\prod_\Omega\Sp(2g+2,\Z/N)
\to
G_N
\to
\PSp(2g+4,\Z/3)
\to
1,
\end{equation}
where $\Omega$ and the action of $\PSp(2g+4,\Z/3)$ are as in
theorem~\ref{thm-hurwitz-monodromy}.
\end{theorem}

\noindent{\bf Question.\ }{\it What happens if $3|N$?\/} The most extreme
case is $G_0$, the case of integer coefficients, which determines $G_N$ for
all $N$.  The congruence subgroup property of $\Sp(2g,\Z)$ probably reduces
this to the determination of $G_{3^n}$ for all $n$.  But the congruence
subgroup property requires $g>1$, so it would only apply for $b\geq8$.

\medskip
Finally, we use theorem~\ref{thm-2} to answer a question of Ellenberg
\cite{Ellenberg}, which we motivate by reinterpreting the Hurwitz monodromy
problem.  If $C\to\PP^1$ is connected and simply branched of degree~4, then
its associated Galois cover has deck group $S_4$.  The Hurwitz monodromy
can be regarded as the action of $\pi_1(\P_b)$ on the family of all Galois
covers of $\PP^1$ that have deck group $S_4$ and satisfy a condition which
is a rephrasing of the simple branching of $C\to\PP^1$.  What makes the
degree-4 case special is that $S_4$ is solvable: it is a semidirect product
$2^2{:}S_3$.  Ellenberg essentially asked: what happens when the $2^2$ is
replaced by the elementary abelian group $p^2$ for some prime $p>3$?  We
show that the resulting monodromy group fits into a split exact sequence
like \eqref{eq-wreath-product}, with $\Z/2$ replaced by $\Z/p$.

Here is a precise formulation of his question, in a more general
context.  Let $X_N$ be the semidirect product $N^2{:}S_3$, with $S_3$
acting by permuting triples of elements of $\Z/N\Z$ with sum $0$.
Consider Galois covers of $\PP^1$ with Galois group $X_N$ and $b$
branch points, such that the small loops around them correspond to
involutions in $X_N$.  When $N$ is even we require further that these
involutions have nontrivial image in $S_3$.  Let $\E_N$ be the set of
isomorphism classes of such covers; this is a local system of finite
sets over $\P_b$, and Ellenberg's question can be phrased: what is the
image $\Gbar_N$ of the monodromy action of $G=\pi_1(\P_b)$ on a fiber
of $\E_N$?  This type of problem was considered by Biggers and Fried
\cite{BF}, who showed that $\Gbar_N$ is transitive on the fiber, so
$\E_N$ is irreducible.  We can go further: for $N$ prime to $3$, we
have completely determined the structure of $\Gbar_N$, except for
$b=4$ or $6$ when $N$ is even.   Theorem~\ref{thm-2} fairly easily implies
the following theorem:

\begin{theorem}
\label{thm-GN}
Suppose $3\nmid N$ and  $b>4$ ($\,b>8$ if $N$ is even).  Then the monodromy
group $\Gbar_N$ of $\E_N\to\P_b$ fits into the split exact sequence
\begin{equation}
\label{eq-GN}
1
\to
\prod_\Omega\PSp(b-4,\Z/N)
\to
\Gbar_N
\to
\PSp(b-2,\Z/3)
\to
1,
\end{equation}
where $\Omega=\PP^{b-3}(\Z/3)$ and  $\PSp(b-2,\Z/3)$ permutes the
factors of the product in the obvious way.
\end{theorem}

\begin{remarks}
The expressions $\Sp(\dots)$ make sense because $b$ always turns out to be
even.  Also, by $\PSp(b-4,\Z/N)$ we mean the quotient of $\Sp(b-4,\Z/N)$ by
its center, which is an elementary abelian $2$-group.
\end{remarks}

The first author is grateful to the University of Michigan, and
especially to Prof. Dolgachev, for organizing support during the fall
of 2007, when this project took form.  The second author is grateful
to Jordan Ellenberg for asking the original question which led to this
project and for pointing out how one can view it as a generalization
of results in \cite{AP} and \cite{H}.

\section{Proof of theorem~\ref{thm-hurwitz-monodromy}}
\label{sec-thm-1}

In this section we will review the relevant results of \cite{EEHS} and then
prove theorem~\ref{thm-hurwitz-monodromy}.  The key feature of the $d=4$
case of the Hurwitz monodromy problem is that $S_4$ is solvable, so that a
degree~4 cover $C\to\PP^1$ determines a number of related covers of
$\PP^1$, shown in Figure~\ref{fig1}.  To organize them we will use
subscripts to indicate their degrees over $\PP^1$.  If $C\to\PP^1$ is
connected and simply branched of degree~4, with $b$ branch points
$p_1,\dots,p_b$, then there is an associated surjection
$\pi_1(C-\{p_i\})\to S_4$, well-defined up to conjugacy by an element of
$S_4$, sending small loops around the $p_i$ to transpositions.  The
corresponding Galois cover $C_{24}$ has $S_4$ as its deck group, and we
define $C_6:=C_{24}/V_4$ and $C_2:=C_{24}/A_4$, where $V_4$ is Klein's
Viergruppe.  $C$ itself is $C_{24}/S_3$ for one of the four conjugate
$S_3$'s in $S_4$, so we could write $C_4$ for $C$.  We will refer to the
covers $C_{24}/D_8\to\PP^1$, for the three conjugate $D_8$'s in $S_4$, as
``the 3 $C_3$'s''.  As explained in \cite[sec.~4]{EEHS}, $C_2$ is
hyperelliptic, $C_2\to\PP^1$ has simple branching over the $p_i$, and
$C_{24}\to C_6$ and $C_6\to C_2$ are unramified with deck groups $2^2$ and
$3$.  The genera of $C_6$ and $C_2$ are $3g+4$ and $g+2$.  Each $C_3$ is
simply branched over $\PP^1$, with $b$ branch points and genus $g+1$. 
These data can be obtained with the Hurwitz formula or by topological
picture-drawing like that in Figure~\ref{fig3}.

\begin{figure}
\[
    \xymatrix{
    & & C_{24} \ar@{-}@/^{7pc}/[ddddd]^{S_4} \\
    & \\
    C=C_4 \ar@{-}@/^{1pc}/[uurr]^{S_3} \ar@{-}@/_{2pc}/[rrddd]
    	& & C_6 \ar@{-}[uu]_{V_4} \\
    & C_3 \ar@{-}@/^{1pc}/[uuur]^{D_8} \ar@{-}[ur]^{\Z/2}
        	&
        	& C_2 \ar@{-}@/_{1pc}/[uuul]_{A_4} \ar@{-}[ul]_{\Z/3} \\
	& \\
    & & \PP^1 \ar@{-}[uuu]_{S_3} \ar@{-}[uul] \ar@{-}[uur]_{\Z/2}
    }
\]
\caption{Covers associated to a degree~4 cover $C\to\PP^1$.}
\label{fig1}
\end{figure}
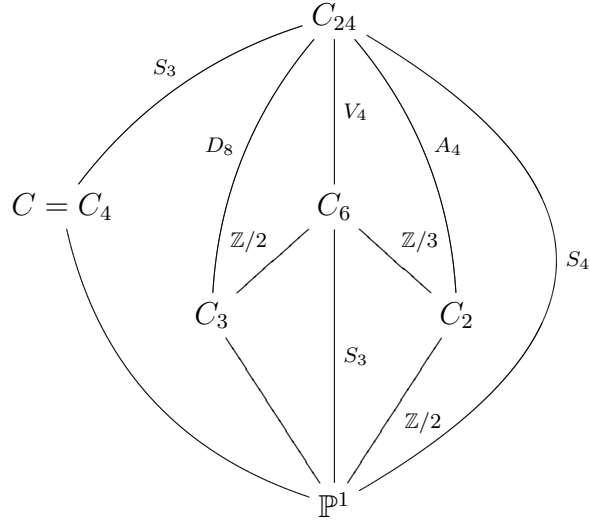

The interplay between these covers allows one to describe the fiber of
$\H_{4,g}\to\P_b$ concretely. Each of the $C_3$'s represents the same point
of $\H_{3,g+1}$, and $C_2$ represents a point of $\H_{2,g+2}$, yielding a
factorization of $\H_{4,g}\to\P_b$ as
$\H_{4,g}\to\H_{3,g+1}\to\H_{2,g+2}=\P_b$.  It is usually more convenient
to work with Galois covers, so we remark that $C_4,C_4'\in\H_{4,g}$ are
equivalent as covers (i.e., are the same point of $\H_{4,g}$) if and only
if the Galois covers $C_{24}$ and $C_{24}'$ are.  This follows from the
conjugacy of index-4 subgroups of $S_4$.  The same argument shows that
$C_3,C_3'\in\H_{3,g+1}$ are equivalent if and only if the Galois covers
$C_6,C_6'$ are.  Because of this, we will sometimes refer to (say) $C_6$ in
order to specify a point of $\H_{3,g+1}$.

Now we describe the fibers of $\H_{2,g+2}$, $\H_{3,g+1}$ and $\H_{4,g}$
over a $b$-tuple $(p_1,\dots,p_b)\in\P_b$ in terms of the possibilities for
the Galois covers $C_2$, $C_6$ and $C_{24}$.  There is only one $C_2$ with
specified branch points $p_1,\dots,p_b$.  The unramified $\Z/3$-covers of
$C_2$ that are Galois over $\PP^1$ are in bijection with the hyperplanes
$h$ in $H_1(C_2;\Z/3)$ that are preserved by the hyperelliptic involution
$\a$ of $C_2$.  The condition that the Galois group be $S_3$ rather than
$\Z/6$ is that $\a$ act on $H_1(C_2;\Z/3)/h$ by negation.  Since $\a$ acts
by negation on all of $H_1(C_2;\Z/3)$, these conditions on $h$ are vacuous,
and the possibilities for $C_6$ are in bijection with $\PP H^1(C_2;\Z/3)$.

Once $C_6\to\PP^1$ is fixed, the possibilities for $C_{24}$ are
parameterized in a similar but more complicated way.  The unramified
covers of $C_6$ with deck group $2^2$ that are Galois over $\PP^1$ are
in bijection with the codimension-two subspaces $L$ of $H_1(C_6;\Z/2)$
which are preserved by $S_3=\Gal(C_6/\PP^1)$.  And the condition for the
Galois group to be $S_4$ rather than some other extension $2^2.S_3$ is
that $S_3$ acts on $H_1(C_6;\Z/2)/L$ in the same way that $S_3=S_4/V_4$
acts on $V_4$.  Dualizing, the choices for $C_{24}$ correspond to the
subgroups $(\Z/2)^2$ of $H^1(C_6;\Z/2)$ which $S_3$ preserves and acts
on by its 2-dimensional irreducible representation, which permutes
triples of elements of $\Z/2$ with sum $0$.  We write $\PP V(C_6)$ for
this set of subspaces, the notation reflecting the fact that it is a
projective space in a non-obvious way.  

To see this, fix one of the three $C_3$'s, and regard $H^1(C_3;\Z/2)$ as
embedded in $H^1(C_6;\Z/2)$ under pullback.
% THE FOLLOWING IS NOT TRUE: this is the same as the subspace of
% $H^1(C_6;\Z/2)$ fixed by one of the three involutions in $S_3$.
Every one of the 2-dimensional subspaces of $H^1(C_6;\Z/2)$ considered
above contains a unique $\Z/2$ lying in $H^1(C_3;\Z/2)$, and every $\Z/2$
in $H^1(C_3;\Z/2)$ lies in a unique one of these 2-dimensional subspaces. 
So $\PP V(C_6)$ is in bijection with $\PP H^1(C_3;\Z/2)$.  The three
$C_3$'s all give the same projective space structure, so the choices for
$C_{24}$, given $C_6$, correspond to points of $\PP V(C_6)\iso
\PP^{2g+1}(\Z/2)$.  We can even be a little fancier and define $V(C_6)$ as
the union of the three $H^1(C_3;\Z/2)$'s, modulo identification under the
group $\Z/3$ of deck transformations.  Then $\PP V(C_6)$ is indeed the
projectivization of $V(C_6)$.

In summary, once $p_1,\dots,p_b$ are fixed, the possibilities for
$C=C_4$ are in bijection with the ordered pairs $(C_6,C_{24})$, where
$C_6$ corresponds to an element of $\PP H^1(C_2;\Z/3)$ and $C_{24}$ to an
element of $\PP V(C_6)$.  All of these constructions can be carried out
simultaneously for all $b$-tuples (this is the basic property of
Hurwitz spaces).  The result is that $\H_{4,g}$ is an unramified cover
of $\P_b$, which factors as $\H_{4,g}\to\H_{3,g+1}\to\P_b$, with a
fiber of the second map parameterizing the possible choices for $C_6$
(or equivalently $C_3$).  The fiber of the first map over a chosen
$C_6$ is $\PP V(C_6)\iso \PP^{2g+1}(\Z/2)$, parameterizing the possible
choices for $C_{24}$, given $C_6$.  So a fiber of $\H_{4,g}$ over
$\P_b$ consists of $|\PP^{2g+3}(\Z/3)|$ many copies of $\PP^{2g+1}(\Z/2)$.

We are interested in the  monodromy action of $G:=\pi_1(\P_b)$ on
this fiber.  It obviously respects the symplectic structure on
$H^1(C_2;\Z/3)$, and the stabilizer of $C_6$ preserves the symplectic
structure on $V(C_6)=H^1(C_3;\Z/2)$.  Therefore the image $G_2$ can be
no larger than in \eqref{eq-wreath-product}. 

Having reviewed the results of \cite{EEHS}, we will now prove the
theorem.  We will write $\b_1,\dots,\b_{b-1}$ for the standard
generators for the spherical braid group on $b$ strands, which is
$G$.

\begin{lemma}
\label{lem-F3-transvection}
The monodromy action of any $\b_i$ on a fiber $\PP H^1(C_2;\Z/3)$ of
$\H_{3,g+1}\to\P_b$ is a symplectic transvection, and $G$ acts by the
full projective symplectic group $\PSp(2g+4,\Z/3)$.
\end{lemma}

\begin{proof}
This is due to Cohen \cite{Cohen}; the key point is the following.
Let $L$ be a simple loop in $\PP^1$ encircling $p_i$ and $p_{i+1}$ but
none of the other branch points.  Then $L$ lifts to a closed loop
$\Ltilde$ on $C_2$.  The monodromy of $\b_i$ on $C_2$ is a Dehn twist
in $\Ltilde$.  (For a visual proof see figs.~5--7 of
\cite[ch.~1]{arnold} and the surrounding text.)  
This acts on cohomology by a transvection.

For the second statement we appeal to Clebsch's theorem
\cite[pp.~224--225]{Clebsch} that $G$ is transitive on the sheets of
$\H_{3,g+1}\to\P_b$, which is to say that it is transitive on $\PP
H^1(C_2;\Z/3)$.  The $G$-conjugates of the $\b_i$  therefore give all the
transvections, which are well-known to generate the symplectic group.
\end{proof}

Now pick a point of $\H_{3,g+1}$; this corresponds to a cover $C_6$
(equivalently, $C_3$) and also to an element of $\PP H^1(C_2;\Z/3)$, say
the one in which $\b_1$ acts by a transvection.  We will abbreviate
$V(C_6)$ to $V$.  Now we consider the subgroup $H$ of $G$ whose monodromy
sends $C_6$ to itself, and the action of $H$ on the fiber $\PP V$ of
$\H_{4,g}$ over $C_6$.  

\begin{lemma}
\label{lem-F2-transvection}
$H$ contains $\b_1$, which acts trivially on $V$, and
$\b_3,\dots,\b_{b-1}$, which act by symplectic transvections.  And $H$
acts on $V$ by the full projective symplectic group
$\PSp(V)\iso\Sp(2g+2,\Z/2)$.
\end{lemma}

\begin{figure}
\begin{center}
\begin{picture}(350,334)(0,0)
\setlength{\unitlength}{1pt}
\thicklines
%\put(0,0){\framebox(350,334)[bl]{}}
% bottom
\put(190,60){\line(1,0){35}}
\put(230,60){\vector(1,0){0}}
% top
\put(190,238){\line(1,0){35}}
\put(230,238){\vector(1,0){0}}
%  left
\put(108,140){\line(0,-1){35}}
\put(108,100){\vector(0,-1){0}}
% right
\put(300,140){\line(0,-1){35}}
\put(300,100){\vector(0,-1){0}}
%%%%%%%%%%%%%%%%%% C6 %%%%%%%%%%%%%%%
% was 140
\put(-4,155){\makebox(0,0)[bl]{%
\setlength{\unitlength}{1.2pt}
\begin{picture}(0,0)(0,0)
\put(0,0){\makebox(0,0)[bl]{\includegraphics[scale=.6]{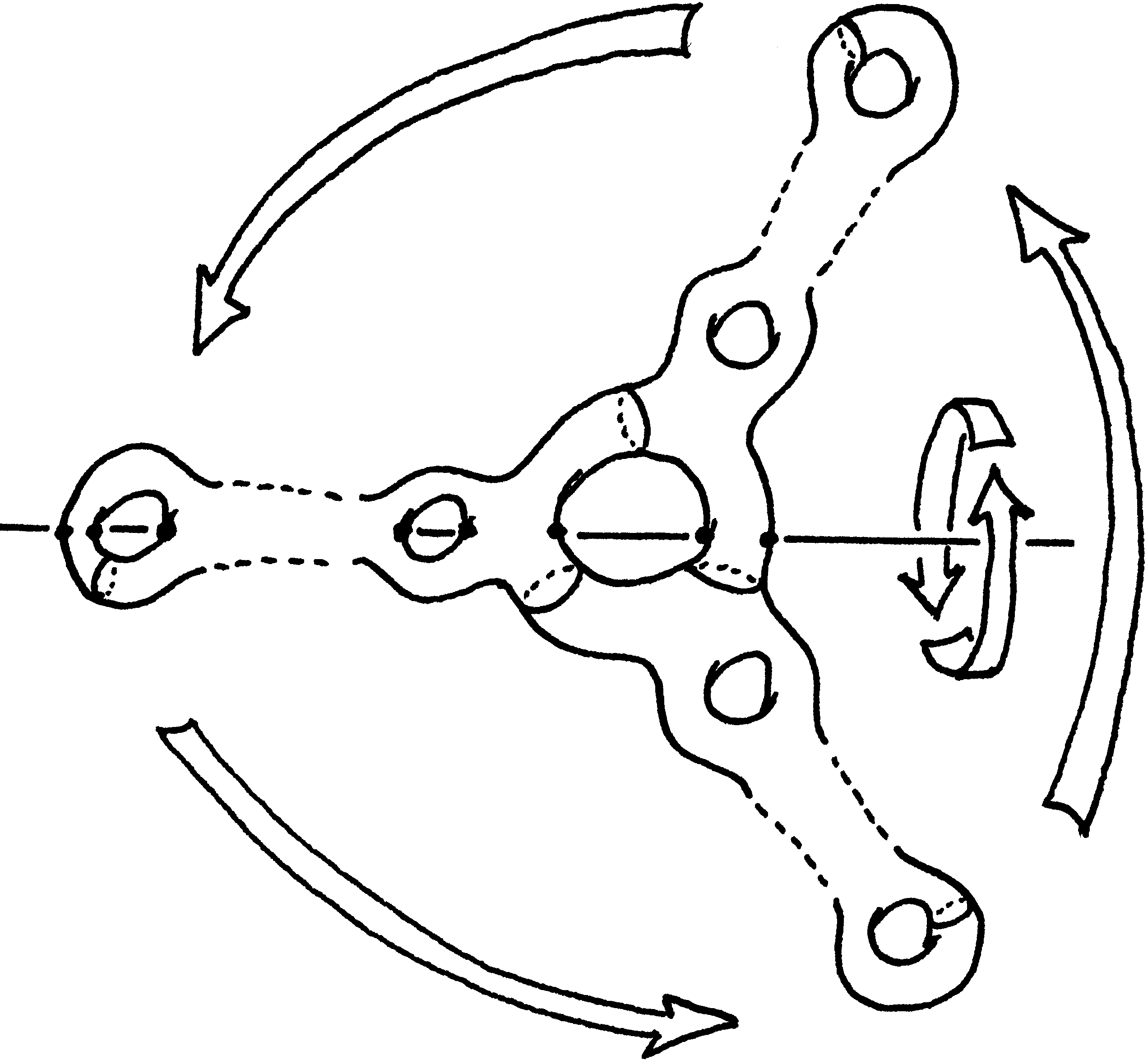}}}
\put(30,130){\makebox(0,0)[c]{$C_6$}}
\put(15,55){\makebox(0,0)[c]{$\b_{b-1}$}}
\put(142,20){\makebox(0,0)[c]{$\b_{b-1}$}}
\put(106,145){\makebox(0,0)[c]{$\b_{b-1}$}}
\put(68,52){\makebox(0,0)[c]{$\b_1$}}
\put(110,64){\makebox(0,0)[c]{$\b_1$}}
\put(80,98){\makebox(0,0)[c]{$\b_1$}}
\end{picture}
}}
%%%%%%%%%%%%%%%%% end C6 %%%%%%%%%%%%
%%%%%%%%%%%%%%%%% C2 %%%%%%%%%%%%%%%%
\put(229,157){\makebox(0,0)[bl]{%
\setlength{\unitlength}{1.2pt}
\begin{picture}(0,0)(0,0)
\put(0,0){\makebox(0,0)[bl]{\includegraphics[scale=.6]{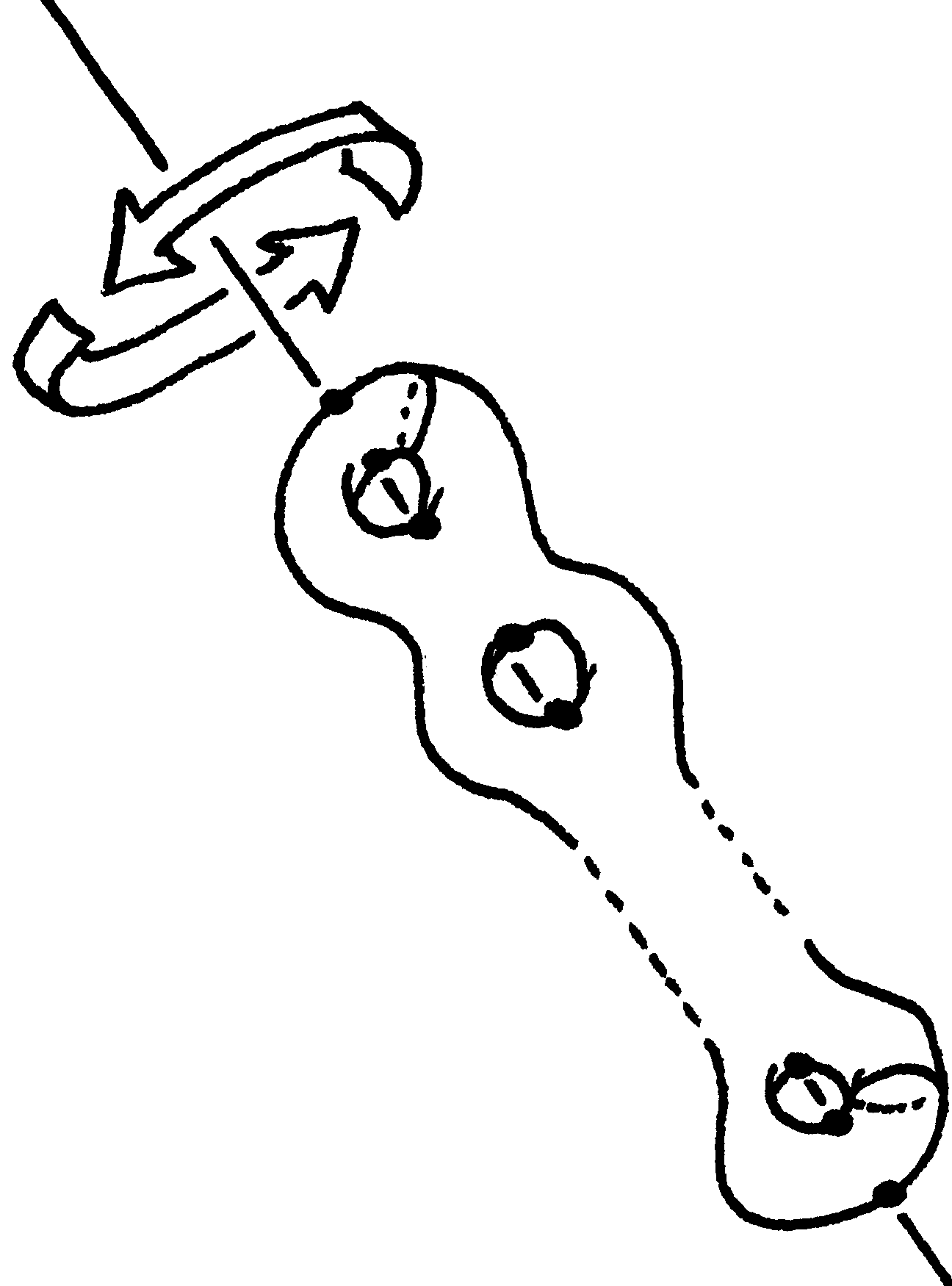}}}
\put(75,60){\makebox(0,0)[c]{$C_2$}}
\put(90,16){\makebox(0,0)[c]{$\b_{b-1}$}}
\put(42,82){\makebox(0,0)[c]{$\b_1$}}
\end{picture}
}}
%%%%%%%%%%%%%%%%% end C2 %%%%%%%%%%%%
%%%%%%%%%%%%%%%%% C3 %%%%%%%%%%%%%%%%
\put(6,5){\makebox(0,0)[bl]{%
\setlength{\unitlength}{1.2pt}
\begin{picture}(0,0)(0,0)
\put(-3,0){\makebox(0,0)[bl]{\includegraphics[scale=.6]{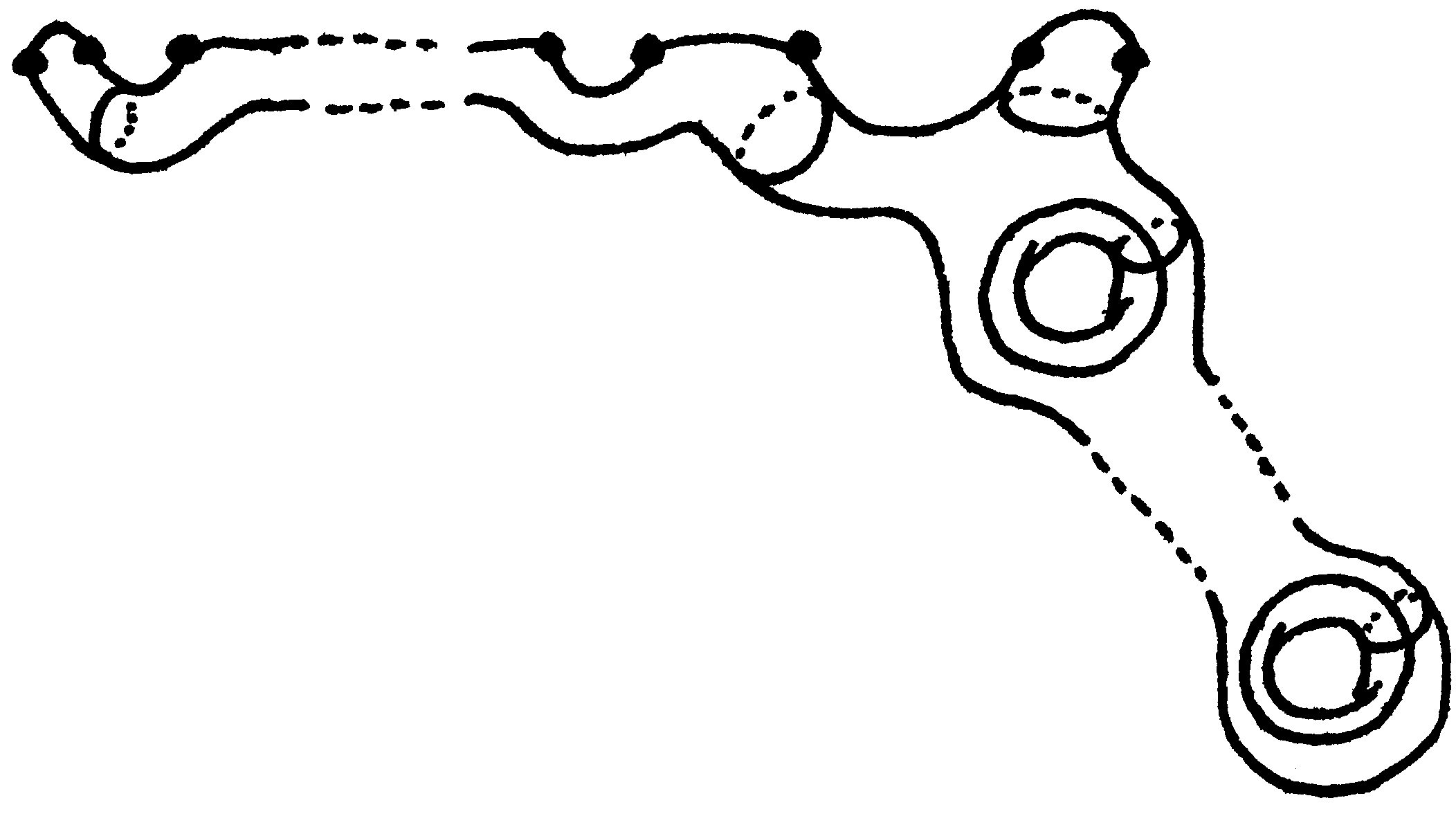}}}
\put(50,30){\makebox(0,0)[c]{$C_3$}}
\put(108,52){\makebox(0,0)[c]{$f_1$}}
\put(81,33){\makebox(0,0)[c]{$e_1$}}
\put(134,20){\makebox(0,0)[c]{$f_{g+1}$}}
\put(96,5){\makebox(0,0)[c]{$e_{g+1}$}}
\end{picture}
}}
%%%%%%%%%%%%%%%%% end C3 %%%%%%%%%%%%
%%%%%%%%%%%%%%%%% P1 %%%%%%%%%%%%%%%%
\put(256,6){\makebox(0,0)[bl]{%
\setlength{\unitlength}{1.2pt}
\begin{picture}(0,0)(0,0)
\put(0,0){\makebox(0,0)[bl]{\includegraphics[scale=.6]{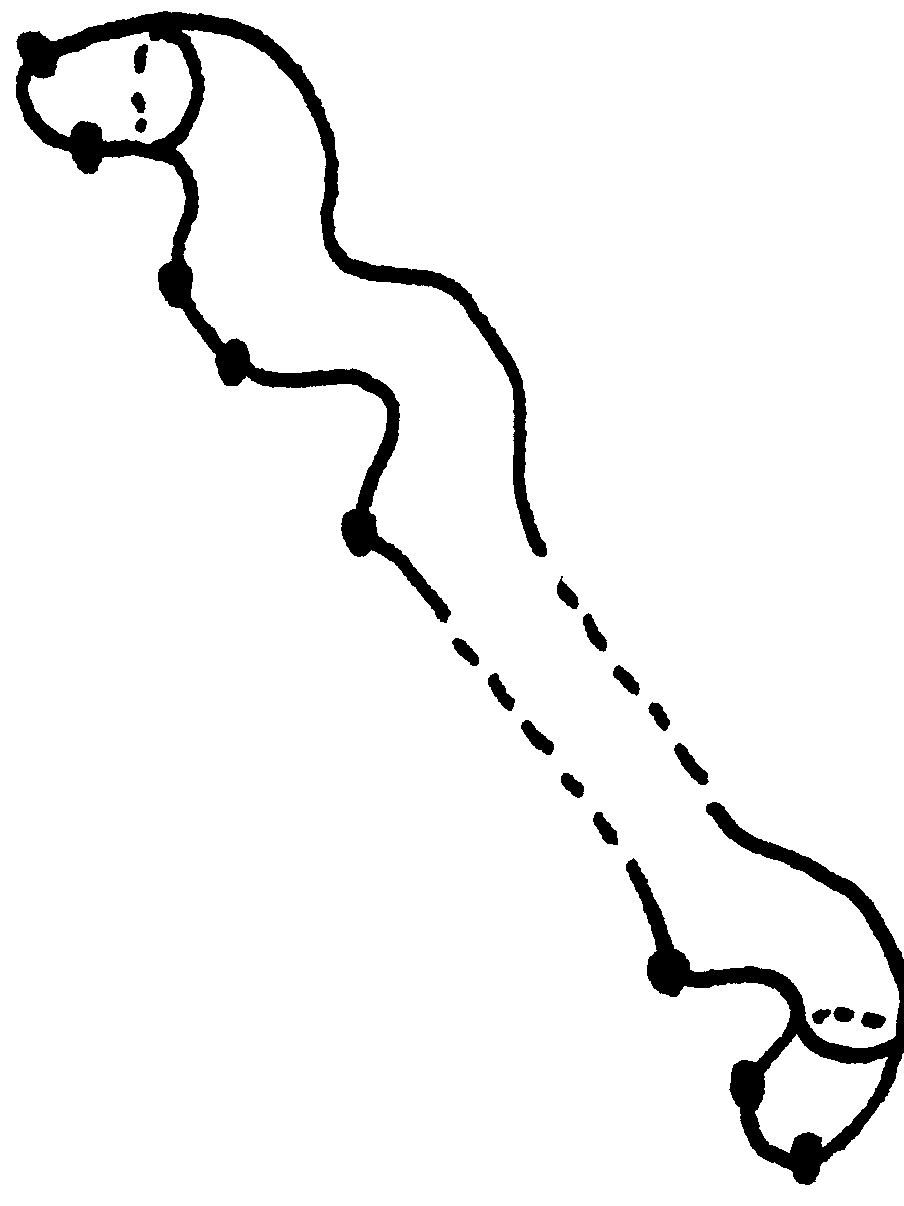}}}
\put(48,60){\makebox(0,0)[c]{$\PP^1$}}
\put(-3,72){\makebox(0,0)[c]{$p_1$}}
\put(0,60){\makebox(0,0)[c]{$p_2$}}
\put(46,-2){\makebox(0,0)[c]{$p_b$}}
\end{picture}
}}
%%%%%%%%%%%%%%%%% end P1 %%%%%%%%%%%%
\end{picture}
\end{center}
\caption{Covers of $\PP^1$ associated to $C\to\PP^1$.}
\label{fig3}
\end{figure}

\begin{proof}
Before beginning the proof proper we make $V$ concrete.  Figure~\ref{fig3}
shows the maps $C_6\to C_3\to\PP^1$ and $C_6\to C_2\to\PP^1$.  The picture
of $\PP^1$ shows the branch points $p_1,\dots,p_b$.  The loop encircling
$p_1$ and $p_2$ has a lift to $C_2$, marked $\b_1$.  We use this notation
because $\b_1$ acts on $C_2$ as a Dehn twist in that loop, which was called
$\Ltilde$ in the proof of lemma~\ref{lem-F3-transvection}.  Now, $C_6$ is
defined as the cover of $C_2$ corresponding to the elements of $\pi_1(C_2)$
having trivial intersection (mod~3) with $\Ltilde$, and is shown.  The deck
group acts by the obvious $\Z/3$ rotation.  Next, there are 3 involutions
in $S_3=\Gal(C_6/\PP^1)$, one of which is the $\Z/2$ rotation around the
horizontal axis.  The quotient $C_3$ is shown, together with the branch
points of $C_6\to C_3$ and a basis $e_1,f_1,\dots,e_{g+1},f_{g+1}$ of
$H^1(C_3)$.  If we indicate lifts of these loops to the 3 `arms' of $C_6$
by $e_j^{(i)}$ and $f_j^{(i)}$, for $i=0,1,2$ and $j=1,\dots,g+1$, then up
to relabeling, the pullback $V^{01}$ of $H^1(C_3)$ is spanned by the
$e_j^{(0)}+e_j^{(1)}$ and $f_j^{(0)}+f_j^{(1)}$.  The other two $C_3$'s
give the same result but with different superscripts.  The space $V$ is the
union of these three vector spaces, modulo cyclic permutation of the upper
labels $0$, $1$ and $2$.

Now, $\b_1$  lies in $H$, so it lifts to $C_6$;  ``the'' action on $C_6$ is
only well-defined up to composition with deck transformations.  But these
act trivially on $V$, by the definition of $V$, so the action of $\b_1$ on
$V$ may be computed from any one of the three lifts of $\b_1$.  One of
these lifts is the composition of the Dehn twists in the three loops marked
$\b_1$.  This obviously leaves  the $e_j^{(i)}$ and $f_j^{(i)}$
unperturbed, so $\b_1$ acts trivially on $V$.

The same analysis applies to $\b_{b-1}$, one of whose lifts to $C_6$
is the composition of the three indicated Dehn twists.  Its
restriction to $V^{01}$ is the transvection in
$f_{g+1}^{(0)}+f_{g+1}^{(1)}$ (with respect to the symplectic form
pulled back from $C_3$, not the one on $H^1(C_6)$).  This
proves that $\b_{b-1}$ acts on $V$ as a transvection.  The argument is
the same for $\b_3,\dots,\b_{b-2}$.

We remark that up to this point, the argument works perfectly well
with $\Z$ coefficients in place of $\Z/2$.

Finally, we again use Clebsch's transitivity theorem, this time
applied to the fibers of $\H_{4,g}\to\P_b$, to deduce that $H$ acts
transitively on the fiber of $\H_{4,g}$ over the point of $\H_{3,g+1}$
corresponding to $C_6$.  This fiber is $\PP V$.  Since the image of $H$
contains a transvection and is transitive on $\PP V$, it contains all
transvections, hence equals $\PSp(V)$.
\end{proof}

Now we will consider the kernel $K$ of $G\to\PSp(2g+4,\Z/3)$ and its
image $K_2$ in $G_2$, which is a subgroup of the direct product
appearing in \eqref{eq-wreath-product}.  We will improve the previous lemma
by showing that $K$ has the same surjectivity properties we just
established for $H$; then we will show that this is a fierce restriction on
$K_2$.  

\begin{lemma}
\label{lem-surjects-to-factors}
The projection of $K_2$ to any factor of 
$\prod_\Omega\PSp(2g+2,\Z/2)$ is surjective.
\end{lemma}

\begin{proof}
Because $G$ permutes the factors transitively, it suffices to treat
any one, say $\PSp(V)$.  Now, $K$ is normal in $H$, and $H$
surjects to $\PSp(V)$, so the image of $K$ is a normal subgroup of
$\PSp(V)$.  It also contains the transvection $\b_{b-1}^3$.  Therefore
it contains all transvections, hence all of $\PSp(V)$.
\end{proof}

If $S$ is a group, then we call a subgroup of a product of copies of
$S$ diagonally embedded if it projects isomorphically to each factor.
The language expresses the fact that it is {\it the\/} diagonal
subgroup, up to automorphisms of the factors.

\begin{lemma}
\label{lem-simple-groups}
Let $S$ be a nonabelian simple group, $\Omega$ a finite set, and
$K_2$ a subgroup of $\prod_\Omega S$ that surjects to each factor.
Then $K_2\iso S^n$ for some $n$, and there is a partition
$\Omega=\Omega_1\coprod\dots\coprod\Omega_n$, such that the $i$th
factor of $K_2$ is diagonally embedded in $\prod_{\Omega_i}S$, for
each $i$.
\end{lemma}

\begin{proof}
We first remark that a product of copies of a nonabelian simple group is a
product in only one way, since the factors are the normal simple subgroups.
We will also use the following standard fact \cite[ch.~2, thm.~4.19]{Su}:
if $A$, $A'$ are groups, then the subgroups $J$ of $A\times A'$ are in
bijection with the 5-tuples $(B,B',C,C',\phi)$ where $B$ and $B'$ are
subgroups of $A$ and $A'$, $C$ and $C'$ are normal subgroups of $B$ and
$B'$, and $\phi$ is an isomorphism $B/C\iso B'\!/C'$.  ($B$ and $B'$ are
the projections of $J$ to $A$ and $A'$,  $C$ and $C'$ are the intersections
of $J$ with the factors, and $J$ is the preimage of the graph of $\phi$
under $B\times B'\to B/C\times B'\!/C'$.)

The proof is by induction on $|\Omega|$, the case of a singleton being
trivial.  So suppose $|\Omega|>1$, choose a point $\w\in\Omega$, and
define $\Omega':=\Omega-\{\w\}$.  We apply the above with
$A=\prod_{\{\w\}}S\iso S$, $A'=\prod_{\Omega'}S$ and $J=K_2\sset
A\times A'$.   By the assumed surjectivity, $B$ surjects to $A$, and
$B'$ surjects to each factor of $\prod_{\Omega'}S$.  By induction,
$B'\iso S^m$ for some $m$, and there is a partition
$\Omega'=\Omega_1'\coprod\dots\coprod\Omega_m'$ such that the $i$th
factor of $B'$ is diagonally embedded in $\prod_{\Omega_i'}S$.  Now,
because $B\iso S$ is simple, $C$ is either all of $B$ or is trivial.
In the first case, $B'\!/C'\iso B/C=1$, so $C'=B'$ also.  Then
$K_2=B\times B'\iso S^{m+1}$, with its $i$th factor diagonally
embedded in $\prod_{\Omega_i}S$, where
$\Omega_1=\Omega_1',\dots,\Omega_m=\Omega_m'$,  $\Omega_{m+1}=\{\w\}$.

In the second case, $B'\!/C'\iso B/C\iso S$, so
$K_2\sset B'\times B=S^m\times S$ is the graph of a surjection $B'\to B$. 
Because $S$ is nonabelian simple, any normal subgroup of $S^m$ is the
product of some of the given factors.  Therefore the kernel of $B'\to B$
consists of $m-1$ factors of $S^m$, say all but the first.  We conclude
that $K_2\sset B'\times B$ is generated by a diagonally embedded copy of
$S$ in each of $\prod_{\Omega_2'}S,\dots,\prod_{\Omega_m'}S$, together with
the graph of an isomorphism from a diagonally embedded copy of $S$ in
$\prod_{\Omega_1'}S$ to $B=\prod_{\{\w\}}S\iso S$.  It follows that
$K_2\iso S^m$, with its $i$th factor diagonally embedded in
$\prod_{\Omega_i}S$, where $\Omega_1=\Omega_1'\cup\{\w\}$ and
$\Omega_2=\Omega_2',\dots,\Omega_m=\Omega_m'$.
\end{proof}

\begin{proof}[Proof of theorem~\ref{thm-hurwitz-monodromy}:]
We will write $S$ for $\PSp(2g+2,\Z/2)$.  We know by
lemma~\ref{lem-F3-transvection} that $G_2$ surjects to $\PSp(2g+4,\Z/3)$,
so to establish the exact sequence it suffices to show that $K_2$ is the
full direct product $\prod_\Omega S$.  Since $g>1$, $S$ is simple.  It
follows from lemmas~\ref{lem-surjects-to-factors}
and~\ref{lem-simple-groups} that there is a partition
$\Omega=\Omega_1\coprod\dots\coprod\Omega_n$ such that $K_2\iso S^n$, its
$i$th factor being diagonally embedded in $\prod_{\Omega_i}S$.  Now, $G$'s
action on $K_2$ permutes the factors of $K_2$, in a manner compatible with
$G$'s action on $\Omega$.  Therefore $G$ respects the partition.  But
$\PSp(2g+4,\Z/3)$ acts primitively on $\Omega$, so either all the
$\Omega_i$ are singletons or else there is only one $\Omega_i$.  In the
first case we have $K_2=\prod_\Omega S$, as desired.  So we must rule out
the case where $K_2$ is isomorphic to $S$ and is diagonally embedded in
$\prod_\Omega S$.  We will do this by exhibiting a nontrivial element of
$K_2$ with trivial projection to one factor.  By
lemma~\ref{lem-F2-transvection}, $\b_1^3$ acts trivially on $V$.  On the
other hand, $\b_1^3$ is $G$-conjugate to $\b_{b-1}^3$, whose image in
$\PSp(V)$ is nontrivial, by the same lemma.

Finally, we show that the sequence \eqref{eq-wreath-product} splits. 
Because $K_2$ has no center, a standard result \cite[ch.~2, thm.~7.11]{Su}
shows that the structure of $G_2$ is determined by the homomorphism
$G_2/K_2\to\Out(K_2)$.  Since $S$ is a nonabelian simple group
with trivial outer automorphism group, $\Out(K_2)=\Sym(\Omega)$.
Also, the homomorphism $\PSp(2g+4,\Z/3)\to\Sym(\Omega)$ is the
permutation action on $\Omega$.  Since there is exists a split
extension giving this homomorphism, and the homomorphism determines
$G_2$, $G_2$ must split.
\end{proof}

In the cases $g=0,1$, lemma~\ref{lem-simple-groups} no longer applies
because the groups $\PSp(2,\Z/2)\iso S_3$ and $\PSp(4,\Z/2)\iso S_6$
are not simple; they are extensions of $\Z/2$ by the simple group
$S'=[S,S]$.  One can describe the permutation representation of
$\pi_1(\P_b)$ on the fiber of $\H_{4,g}\to\P_b$ in a manner suitable
for computer calculation, and for $g=0$ we discovered
$|G_2|=3^{40}.2^{16}|\PSp(4,\Z/3)|$, so $K_2=3^{40}.2^{16}\sset
S_3^{40}$.  For $g=1$ the calculation exceeded our available computing
power, so we proceeded as follows.  An argument as in the proof of
theorem~\ref{thm-hurwitz-monodromy} shows that
$K_2':=K_2\cap\prod_\Omega S'$ is either the full direct product
$\prod_\Omega S'$ or is isomorphic to $S'$ and is diagonally embedded
in $\prod_\Omega S'$.  ($K_2'$ turns out to be the commutator subgroup
of $K_2$, justifying the notation.  This also holds in the $g=0$
case.)  A computer-aided calculation shows that $K_2'$ is the full direct
product $\prod_\Omega S'$.  The crucial step is an analogue of
lemma~\ref{lem-surjects-to-factors} for $K_2'$.  Namely, while $\b_i^3$
lies in $K_2$, it does not lie in $K_2'$ because transvections lie outside
$S'$.  Nonetheless, an explicit calculation shows that $[\b_1^3,\b_2^3]$ is
a non-trivial element of $K_2'$ which projects trivially to at least one
factor $S'$, hence $K_2'=\prod_\Omega S'$.  

As described below, we wrote down an explicit faithful permutation
representation of $G_2/K_2'$, which was within reach of computer
calculation.   We found that $G_2/K_2'$ is $2^{16}.\PSp(4,\Z/3)$ for $g=0$
and $2^{168}.\PSp(6,\Z/3)$ for $g=1$.  Although we already knew this when
$g=0$, in this representation we could show that  extension is not split,
which was out of reach before killing $K_2'$.  We did not apply sufficient
computing power to determine whether or not it splits for $g=1$.  We
carried out our computer calculations using GAP \cite{GAP}.

To describe our representation of $G_2/K_2'$ we recall from
\cite[Section~1]{EEHS} the (faithful) permutation representation of $G_2$
on the collection $\Sigma$ of $S_4$-orbits of $b$-tuples
$(\sigma_1,\ldots,\sigma_b)$ of 2-cycles in $S_4$ such that
$\sigma_1\cdots\sigma_b=1$ and
$\langle\sigma_1,\ldots,\sigma_b\rangle=S_4$.  Here, $\b_i$ acts by
replacing $\sigma_i$ by $\sigma_{i+1}$ and $\sigma_{i+1}$ by
$\sigma_{i+1}^{-1}\sigma_i\sigma_{i+1}$ and leaving all other
$\sigma_j$ invariant, and $S_4$ acts by simultaneous conjugation on
all elements of a $b$-tuple.

In a similar fashion we may identify $\Omega$ with the $S_3$-orbits of
$b$-tuples of 2-cycles in $S_3$ so that if we fix an isomorphism
$S_3\iso S_4/V_4$, then the induced map $\Sigma\to\Omega$ is
$G_2$-equivariant.  If we fix $\omega\in\Omega$ to be the point
corresponding to $C_6$ and write $\Sigma_\omega$ for the fiber over
$\omega$, then we may identify $\Sigma_\omega$ with $\PP V$ and
$S=\PSp(V)$ with the factor of $\prod_\Omega S$ over $\omega$.

If we write $H_2$ for the stabilizer in $G_2$ of $\omega$, then the
representation $G_2\to\Sym(\Omega)$ is equivalent to the left
representation of $G_2$ on $G_2/H_2$.  Moreover, if we write $H_2^*$
for the kernel of the composite map $H_2\to S\to\Z/2$, then $K_2'$ is
the intersection of all $G_2$-conjugates of $H_2^*$ and hence is the
kernel of the left representation of $G_2$ on $\Omega'=G_2/H_2^*$.  In
particular, given a set of coset representatives of $G_2/H_2^*$ and a
black box for identifying when two elements of $G_2$ lie in the same
coset, it is easy to compute the representation $G_2\to\Sym(\Omega')$:
$\b_i$ takes the coset $\alpha_jH_2^*$ to the coset
$\b_i\alpha_jH_2^*$.

To construct representatives one takes a known subset
$\alpha_1,\ldots,\alpha_m$, computes $\b_i\alpha_jH_2^*$ for
$i=1,\ldots,b-1$ and $j=1,\ldots,m$, adds any new cosets which arise to the
known subset, and repeats until no new cosets are constructed.

To construct the black box observe that the elements $\gamma_1,\gamma_2$
represent the same coset if and only if $\gamma=\gamma_1^{-1}\gamma_2$
lies in $H_2^*$, and the latter occurs if and only if $\gamma$ both
stabilizes $\omega$ and lies in the kernel of $H_2\to\Z/2$.  For $g=1$, the
parity map $S\simeq S_6\to\Z/2$ is not the restriction of the parity map
$\Sym(\Sigma_\omega)\to\Z/2$, hence a little work is required to determine
the former; a transvection in $S$ fixes $2^3-1$ lines and permutes the
other $2^4-2^3$ in pairs, hence is a product of four 2-cycles in
$\Sym(\Omega_\omega)$.  If we chose $\omega$ to correspond to the $b$-tuple
$(\sigma_1,\ldots,\sigma_b)$ with $\sigma_1=\sigma_2=(12)$ and
$\sigma_3=\cdots=\sigma_b=(23)$, then one can easily verify that
$\b_3,\ldots,\b_{b-1}$ stabilize $\omega$, they each act non-trivially on
$V$, and they generate $S$.  Since they generate $S$ and are conjugate,
they must all map to the nontrivial element of $\Z/2$ and the image of
$s\in S$ under $S\to\Z/2$ is the parity of the length of $s$ as a product
in $\b_3,\ldots,\b_{b-1}$.

\section{Proof of theorems~\ref{thm-2} and \ref{thm-GN}}
\label{sec-thm-2}

We first introduce a little notation for talking about $\V_N$.  Choosing a
point of $\H_{3,g+1}$ means choosing a simply branched cover $C_3\to\PP^1$,
or equivalently the associated Galois cover $C_6\to\PP^1$.  We define
$V_N(C_6)$ to be $H^1(C_3;\Z/N)$, or more intrinsically as the union of the
pullbacks to $H^1(C_6;\Z/N)$ of the three $H^1(C_3;\Z/N)$'s, modulo
identifications by the action of $\Z/3$.  For fixed
$(p_1,\dots,p_b)\in\P_b$, the fiber of $\V_N$ is $\oplus_{C_3}V_N(C_6)$,
where the sum extends over the points $C_3\in\H_{3,g+1}$ lying above
$(p_1,\dots,p_b)$.  When $N=2$, $V_2(C_6)$ is just $V(C_6)$ from
section~\ref{sec-thm-1}, giving the relation to the Hurwitz monodromy
problem.

Now we can discuss the monodromy.  The map $G\to\PSp(2g+4,\Z/3)$ is the
same as in the previous section, corresponding to the action on $\Omega=\PP
H^1(C_2;\Z/3)$.  As before, we write $K$ for the kernel, which acts on
$\V_N$ by a subgroup of $P_N:=\prod_\Omega\Sp(2g+2,\Z/N)$.  Also, we saw in
lemma~\ref{lem-F3-transvection} that $\b_1$ acts on $H^1(C_2;\Z/3)$ as a
transvection, so it distinguishes an element of $\Omega$.  We write $H$ for
the $G$-stabilizer of this point, $C_6$ for the corresponding $S_3$-cover
of $\PP^1$, and $V_N$ for $V_N(C_6)\iso(\Z/N)^{2g+2}$.

\begin{lemma}
\label{lem-foo-1}
If $g\geq 0$ and $N\geq 0$, then $H$ acts on $V_N$ as $\Sp(V_N)$.
\end{lemma}

\begin{proof}
It suffices to prove this in the case $N=0$, i.e., with $\Z$ coefficients. 
A'Campo \cite[Thm.~1(2)]{ACampo} studied a particular representation of the
braid group $B_{\mu+1}$, $\mu$ even, into $\Sp(\mu,\Z)$.  We have a
representation of $B_{b-2}=\langle \beta_3,\dots,\beta_{b-1}\rangle\sset H$
into $\Sp(2g+2,\Z)$.  (Recall that $b=2g+6$.) In both cases the braid
generators act by transvections in primitive lattice vectors (this uses
lemma~\ref{lem-F2-transvection}, whose proof goes through perfectly well
over $\Z$).  These representations are essentially unique, since the
transvections in two nonproportional vectors braid if and only if pairing
the vectors yields $\pm1$.  Therefore our representation contains his, with
$\mu=2g+2$, and we even have an extra generator.  He proves that the image
of his representation contains the level-$2$ congruence subgroup of
$\Sp(2g+2,\Z)$, so the image of ours does too.  (One can show that our
extra generator doesn't enlarge the image of the representation.)

Since the image of $H$ contains the level-$2$ congruence subgroup of
$\Sp(2g+2,\Z)$, all we have to show is that $H$ surjects to
$\Sp(2g+2,\Z/2)$.  We did this in lemma~\ref{lem-surjects-to-factors}.
\end{proof}

\begin{lemma}
\label{lem-foo-2}
If $g\geq0$ and $3\nmid N$, then the projection of $K$ to any factor of
$P_N=\prod_\Omega\Sp(2g+2,\Z/N)$ is surjective.
\end{lemma}

\begin{proof}
Follow the proof of lemma~\ref{lem-surjects-to-factors}; the only
modification needed is that depending on one's definition of a
transvection, $\b_{b-1}^3$ might not be one (e.g.~if 3 is not a square in
$\Z/N$).  But regardless of this choice of definition, if $m\geq 1$
satisfies $3m\equiv 1\pmod{N}$, then the cyclic group $\b_{b-1}^3$
generates contains the transvection $\b_{b-1}^{3m}$.  
\end{proof}

\begin{proof}[Proof of theorem~\ref{thm-2}:]
It suffices to prove that $K$ surjects to $P_N$, and by the Chinese
remainder theorem it suffices to treat the case where $N$ is a prime power
$p^n$.  First we treat the case $N=p$.  Under our hypothesis on $g$,
$\PSp(2g+2,\Z/p)$ is a nonabelian simple group.  Then the argument for
theorem~\ref{thm-hurwitz-monodromy} implies that $K$ surjects to the
central quotient $\prod_\Omega\PSp(2g+2,\Z/p)$ of $P_p$.  Since
$\Sp(2g+2,\Z/p)$ is a nonsplit extension of $\PSp(2g+2,\Z/p)$, $K$ surjects
to $P_p$.  

Now we suppose $N=p^n$ for $n>1$.  We write $\G$ for the level $p^{n-1}$
congruence subgroup of $\Sp(2g+2,\Z/p^n)$ and assume inductively that $K$
surjects to $P_{p^{n-1}}$.  So we must show that $G_N\cap\prod_\Omega\G$ is
all of $\prod_\Omega\G$.  Now, $\G$ is an elementary abelian $p$-group, and
the action of $\Sp(2g+2,\Z/p^n)$ on it factors through $\Sp(2g+2,\Z/p)$,
whose action on $\G$ is equivalent to the adjoint action on
$\mathfrak{sp}(2g+2,\Z/p)$.  First we suppose $p>2$, so that this action is
irreducible.  Observe that the action of $P_p$ on $\prod_\Omega\G$ is by
the direct sum of $|\Omega|$ many distinct irreducible representations of
$P_p$.  Since $G_N$ surjects to $P_p$, $G_N\cap\prod_\Omega\G$ is an
invariant subspace, so it is the product of some of the factors of
$\prod_\Omega\G$.  It also surjects to each factor, by
lemma~\ref{lem-foo-2}, so it must be the product of all of them.  This
finishes the proof for $p\neq2$.

The same argument works for $p=2$, even though $\G$ is no longer
irreducible under $\Sp(2g+2,\Z/2)$.  The scalar matrix $1+2^{n-1}$ in
$\G\iso\mathfrak{sp}(2g+2,\Z/2)$ is invariant, the quotient by the span of
this vector is irreducible, and there is no invariant complement.  This
last property is key, because it implies that the only $P_p$-invariant
subspace of $\prod_\Omega\G$ that projects onto each factor is the whole
product.  So the argument still applies.
\end{proof}

Now we explain the application of theorem~\ref{thm-2} to Ellenberg's
question.  As in the previous section, we will indicate degrees of
covers of $\PP^1$ by subscripts.  Suppose $3\nmid N$ and
$C_{6N^2}\in\E_N$, i.e., $C_{6N^2}$ is a Galois cover of $\PP^1$ with
Galois group $X_N=N^2{:}S_3$ and $b$ branch points, and the small
loops around them permute the sheets by involutions in $X_N$ with
nontrivial image in $S_3$.  Then $b$ must be even because the product
of the $b$ loops in $\Z/2=S_3/3$ must be trivial.  Analogously to
section~\ref{sec-thm-1}, we define $C_6$ as $C_{6N^2}/N^2$, $C_2$ as
$C_6/3$ and the three $C_3$'s as the quotients of $C_6$ by the 3
involutions in $S_3$.  These are exactly the same covers we met in
section~\ref{sec-thm-1} and they fit into a diagram similar to
Figure~\ref{fig1}.

We define the projectivization $\PP V_N(C_6)$ as the set of direct summands
$\Z/N$ of $V_N(C_6)$.  The arguments of section~\ref{sec-thm-1}, with
$\Z/N$ in place of $\Z/2$, imply that $C_{6N^2}$ corresponds to an element
of $\PP V_N(C_6)$, and that the fiber of $\E_N\to\P_b$ over $p\in\P_b$ is
in bijection with the set of pairs $(C_6,C_{6N^2})$, where $C_6$
corresponds to an element of $\Omega=\PP H^1(C_2;\Z/3)\iso \PP^{b-3}(\Z/3)$
and $C_{6N^2}$ to an element of $\PP V_N(C_6)$.  That is, the fiber is
$\coprod_\Omega \PP^{b-5}(\Z/N)$.  It is clear that the action of $G$ on
this set is determined by its action on $\oplus_\Omega V(C_6)$, which is
exactly the fiber of $\V_N$.  Indeed, the action is given by projectivizing
the action on each summand, so the monodromy group $\Gbar_N$ is got from
\eqref{eq-foo} by replacing $\Sp$ by $\PSp$.  

\begin{proof}[Proof of theorem~\ref{thm-GN}:]
We have already explained why $\Gbar_N$ is the quotient of $G_N$ by the
center of $K_N=\prod_\Omega\Sp(b-4,\Z/N)$, so all we have to do is show
that the sequence splits.  By the Chinese remainder theorem, it suffices to
treat the case with $N$ a prime power $p^n$.  We appeal to a theorem of
Gross and Kov\'acs \cite[Cor.~4.4]{GK} which describes the structure of
extensions like \eqref{eq-GN} in terms of the stabilizer of one factor of
the product.  We fix $\w\in\Omega$ and let $\Hbar_N\sset \Gbar_N$ be its
stabilizer.  Their result asserts that \eqref{eq-GN} splits if and only if
$$
1
\to
\prod_{\w'\in\Omega}S\!\biggm/\!\prod_{\w'\neq\w}S
\to
\Hbar_N\!\!\biggm/\!\prod_{\w'\neq\w}S
\to
\Hbar_N\!\!\biggm/\!\prod_{\w'\in\Omega}S
\to
1,
$$
does, where $S=\PSp(b-4,\Z/p^n)$.  This sequence has the form
$$
1
\to
\PSp(b-4,\Z/p^n)
\to
\ ?\ 
\to
3{\cdot}3^{b-4}{:}\Sp(b-4,\Z/3)
\to
1,
$$
the right term being a maximal parabolic subgroup of $\PSp(b-2,\Z/3)$.
Since the left term is centerless, the structure of the extension
is given by the natural homomorphism from the right term to 
$\Out(S)$, which is solvable.  Since the right term is perfect,
this map is trivial, so the sequence splits, so \eqref{eq-GN} does too.

($\Out(S)$ is known exactly, cf. \cite{petechuk} for the case $b\geq10$ and
\cite{abe} for the case $b\geq6$ with $N$ odd.  But it is much easier to
see solvability than to work the group out exactly.)
\end{proof}

\begin{remark}
Since we know $\Gbar_N$, we recover the result of Biggers and Fried
\cite{BF} that $G$ is transitive on the fiber of $\E_N\to\P_b$, which is
the same as the irreducibility of $\E_N$.  On the other hand, when $N\neq0$
one could use their result to prove an analogue of lemma~\ref{lem-foo-1}
without relying on A'Campo's theorem.  Namely, $H$ acts on $\PP V$ as
$\PSp(V_N)$; one mimics the proof of lemma~\ref{lem-F2-transvection}, using
their transitivity result in place of Clebsch's.  One can then use this to
prove lemma~\ref{lem-foo-1} itself (for $N\neq0$).
\end{remark}

\end{document}